\documentclass [11pt] {article} 
\usepackage[affil-it]{authblk}

\usepackage{amsrefs} 
\usepackage{subcaption}

\usepackage{graphicx}
\usepackage{amssymb}
\usepackage{amsmath}
\usepackage[section]{placeins}
\usepackage{mathrsfs}  
\usepackage{multirow}
\usepackage{amsthm}
\usepackage{float}
\usepackage{bbm}
\usepackage{colortbl}
\usepackage{xcolor}

\usepackage{mwe}
\usepackage{mathtools}
\usepackage{hyperref}
\newtheorem{thm}{Theorem}[section]
\newtheorem{prop}[thm]{Proposition}
\newtheorem{lem}[thm]{Lemma}

\newtheorem{rem}[thm]{Remark}
\newtheorem{cor}[thm]{Corollary}
\newtheorem{ex}[thm]{Example}

\newcommand{\be}{\begin{equation}}
\newcommand{\ee}{\end{equation}}

\usepackage{datetime} 
\usepackage[T1]{fontenc}
\usepackage[english] {babel}

\let\Re=\undefined\DeclareMathOperator{\Re}{Re}

\newcommand{\wh}[1]{\widehat{#1}}
\newcommand{\ol}[1]{\overline{#1}}
\newcommand\Z{\mathbb{Z}}
\newcommand\R{\mathbb{R}}
\newcommand{\mc}[1]{\mathcal{#1}}
\newcommand{\del}{\delta}
\newcommand{\op}{\textup{op}}

\begin{document}

\baselineskip=15pt

\title{On the convergence of explicit formulas for $L^2$ solutions to the Benjamin--Ono and continuum Calogero--Moser equations}

\author{Yvonne Alama Bronsard, Thierry Laurens}

\author{
Yvonne Alama Bronsard%
\thanks{Massachusetts Institute of Technology, Cambridge, MA 02139, 
\texttt{yvonneab@mit.edu}}, \  
Thierry Laurens%
\thanks{University of Wisconsin--Madison, Madison, WI, 53706,
\texttt{laurens@math.wisc.edu}}}

\maketitle

\begin{abstract}
By developing discrete counterparts to recent advances in nonlinear integrability, and in particular to the discovery of explicit formulas, we design and analyze fully-discrete approximations to the Benjamin--Ono (BO) and  continuum Calogero--Moser (CCM) equations on the torus. We build on the key observation that discretizing such explicit formulas yields schemes that are
exact in time (requiring only spatial discretization) and have a computational cost independent of the final time $T$.
In this work, we first generalize the fully-discrete schemes of \cite{ABCD24} to include numerical approximations with better structure preservation properties, including the conservation of mass and momentum in the case of the (BO) equation. Secondly, building on recent analyses of the corresponding Lax operators, we extend the convergence results to this class of schemes for rough solutions $u(t)$ merely belonging to $L^2(\mathbb{T})$ for (BO) and $L^2_{+}(\mathbb{T})$ for (CCM), the latter  of which is precisely the scaling-critical regularity. Our main theorem states that the $L^2(\mathbb{T})$-norm of the error goes to zero as the truncation parameters go to infinity, uniformly on any bounded time interval $[-T,T]$. 
As an example, we apply our scheme to the (BO) equation with a square-wave initial profile, and obtain the first numerical evidence of the Talbot effect for (BO) supported by a rigorous convergence result.
\end{abstract}



\section{Introduction}

 This work pertains to two systems.  The first
is the Benjamin--Ono equation
\begin{equation}\label{eq:BO}\tag{BO}
\partial_t u(t,x) = \partial_x \left( |\partial_x| u - u^2 \right)(t,x), \quad  u_{|t=0}(x) = u_0(x), \quad (t,x) \in \mathbb{R} \times \mathbb{T}, 
\end{equation}
\renewcommand*{\theHequation}{notag.\theequation}%
where $u(t,x)\in \mathbb{R}$ is a real-valued solution, and the nonlocal operator $|\partial_x|$ is defined in Fourier space as
\[
\widehat{|\partial_x|f}(k) = |k|\widehat{f}(k), \quad f\in L^2(\mathbb T).
\]
This is a nonlocal, integrable, dispersive model introduced by Benjamin \cite{BO} and Davis--Acrivos \cite{davis1967solitary} to describe internal waves in a deep stratified fluid, and was later popularized by Ono \cite{ono1975algebraic}.
In this work we study numerical approximations of \eqref{eq:BO} for initial data $u_0 \in L^2(\mathbb{T})$. In this setting, {global well-posedness was first proved by Molinet \cite{Molinet08}} using Tao’s gauge transform combined with Bourgain space estimates, and was revisited in Molinet--Pilod \cite{Molinet12}. Sharper results were later obtained by Gérard–Kappeler–Topalov \cite{Gerard23WP}, who exploited the complete integrability of the equation and, in particular, a Birkhoff normal form transformation: they showed that \eqref{eq:BO} is globally well-posed in $H^s$, $s>-\frac12$, and ill-posed otherwise. We note that the analogous $L^2$ well-posedness result on the line was established by Ionescu--Kenig \cite{IK} and revisited in \cites{Molinet12,IT,Talbut}, and sharp well-posedness was obtained by Killip--Laurens--Vi\c{s}an \cite{Killip24}.
For a survey of earlier results on \eqref{eq:BO}, we refer to the book of Klein--Saut \cite[Chapter 3]{KS-book-BO}, and for the rigorous derivation of \eqref{eq:BO} as an internal water wave model to Paulsen \cite{P24-BO}.
\newline

The second equation we consider is the focusing ($+$ sign) or defocusing ($-$ sign) continuum Calogero--Moser equation, also known as the Calogero–Sutherland derivative nonlinear Schr\"odinger (DNLS) equation,
\begin{equation}\label{eq:CS}\tag{CCM}
i\partial_t u + \partial_x^2u \pm \frac{2}{i} u\partial_x \Pi(|u|^2) = 0,\quad  u_{|t=0}(x)= u_{0}(x), \quad (t,x) \in \mathbb{R} \times \mathbb{T}. 
\end{equation}%
\renewcommand*{\theHequation}{notag2.\theequation}%
The focusing equation was derived by Abanov--Bettelheim--Wiegmann \cite{abanov2009integrable} as a continuum limit of the Calogero--Moser particle system, and the defocusing model was introduced by Pelinovsky \cite{pelinovsky1995intermediate} to describe the modulation of an approximately monochromatic wave-packet solution to \eqref{eq:BO}.  Here, the Riesz--Szeg\H o projector $\Pi$ is defined in Fourier space as 
\be\label{eq:Pi}\tag{$\Pi$}
\widehat{\Pi f}(k) = \mathbbm{1}_{k\ge 0} \, \widehat{f}(k), \quad f\in L^2(\mathbb T).
\ee
\renewcommand*{\theHequation}{notag3.\theequation}%
This is an orthogonal projection on $L^2$, and the corresponding closed subspace is the Hardy space
\begin{equation}\label{eq:hardy}
L^2_+ = \{ f\in L^2 : \wh{f}(k) = 0 \text{ for }k<0 \} .
\end{equation}
In the work \cite{Badreddine24}, Badreddine shows that the \eqref{eq:CS} equation is globally well-posed at the scaling-critical regularity, namely in $L^2_{+}$, with the requirement
\[
\|u_0\|_{L^2}^2 = \frac{1}{2\pi} \int_{-\pi}^\pi |u_0|^2\, dx < 1
\]
in the focusing case.  As explained in~\cite{Badreddine24}, this threshold corresponds to the $L^2$-norm of certain traveling wave profiles; however, these profiles do not generate all of the traveling wave solutions on $\mathbb{T}$, which in general can have arbitrary $L^2$ norm \cite{Badreddine25}.  This threshold for well-posedness in the focusing case is most likely sharp:   on the line $\mathbb{R}$,  Hogan--Kowalski \cite{HK24} exhibited smooth solutions that blow up in (at most) infinite time, and later Kim--Kim--Kwon~\cite{Kim24} constructed finite-time blow-up solutions with mass arbitrarily close to the threshold\footnote{On the line $\mathbb{R}$, the threshold for well-posedness is $\| u_0\|_{L^2(\mathbb{R})} =  \int_{\mathbb{R}} |u_0|^2 \le 2\pi$.}.
\newline

In this work we develop discrete counterparts to recent advances in nonlinear integrability.
First, we discuss the {\it design} of the schemes, which builds on the discovery of {\it explicit formulas} for certain nonlocal integrable PDEs. 
Explicit formulas on the torus $\mathbb{T}$ were established by Gérard \cite{G} for the \eqref{eq:BO} equation and by Badreddine \cite{Badreddine24} for the \eqref{eq:CS} equation.
We refer to Remark~\ref{rem:szego} for a detailed discussion on the cubic Szeg\H{o} equation, the first equation for which an explicit formula was found \cite{gerard2015explicit} and, to date, the only other equation on the torus with a known explicit formula.

These explicit formulas have had remarkable consequences, shedding light on the global well-posedness and qualitative behavior of the underlying flow. 
In particular, they have enabled global well-posedness results and the characterization of the traveling waves for the \eqref{eq:CS} equation \cite{Badreddine24, Badreddine25}, as well as the analysis of zero-dispersion limits for both \eqref{eq:CS} and \eqref{eq:BO}, see the works of Badreddine \cite{badreddine2024zero}, Gassot \cite{Gassot}, and M{\ae}hlen \cite{maehlen2025zero}.
In the case of the real line $\mathbb{R}$, analogous formulas have also proved to be an invaluable tool for a wide variety of problems. 
For \eqref{eq:BO}, this includes the zero-dispersion limit \cite{GGM24b,Gerard25}, multisoliton dynamics \cite{Sun}, the evolution of rational initial data \cite{GGM24a}, and, very recently, the soliton resolution conjecture, solved by Gérard--Gassot--Miller \cite{GGM26} for sufficiently regular, decaying solutions.  
For \eqref{eq:CS} on $\mathbb{R}$, also known as the Calogero--Moser DNLS, these formulas have played a central role in the analysis of well-posedness \cite{KLV25}, the establishment of scattering results \cite{chen2025scattering}, as well as in the study of turbulent behavior \cite{HK24}.

On the discrete side, these formulas have also proven valuable: the study of Alama Bronsard--Chen--Dolbeault \cite{ABCD24} was the first to employ them for numerical approximation, providing significant insight into the long-time dynamics of the equations through quantitative error bounds.

In the next section, we present these formulas and outline the key ideas behind their numerical approximation.

\subsection{Explicit formulas}
In the definition of the explicit formulas, three essential operators arise: the Riesz--Szeg\H{o} operator defined in \eqref{eq:Pi}, and the shift $S^*$ and Lax operator $L_{u_0}$ which are well-defined on the Hardy space $L^2_{+}$ from \eqref{eq:hardy}, as presented next. 

First, let $S^* = \Pi (e^{-ix}\ \cdot)$ denote the operator on $L^2_+$ corresponding to the left-shift in frequency variables:
\[
\wh{S^*f}(k) = \mathbbm{1}_{k\ge 0}\, \wh{f}(k+1), \quad f \in L^2_{+}.
\]
Next, the Lax operators for \eqref{eq:BO} and \eqref{eq:CS} are semi-bounded self-adjoint operators defined on $H^1_{+} = H^1 \cap L^2$ by
\begin{equation}\label{eq:lax-ops}
 L_{u_0}^{\text{BO}}f = -i\partial_xf - \Pi(u_0f) \quad {\text and} \quad
L_{u_0}^\text{CS} f = -i\partial_xf \mp u_0 \Pi(\overline{u_0}f),
\end{equation}
which are well defined for $u_0 \in L^2$ and $u_0 \in L^2_{+}$, respectively; see Proposition \ref{prop:lax} and \ref{prop:lax'}.

Once mapped into frequency space, Gérard's \cite{G} and  Badreddine's \cite{Badreddine24} explicit formulas are given as follows: for $k \ge 0$,
\begin{align}\label{eq:explicitForm}
\widehat{u}_{\text{BO}}(t,k) =\left\langle (e^{it(I+2L_{u_0}^{\text{BO}})}S^{*})^k \Pi u_0, 1 \right\rangle \quad \text{and}\quad
\widehat{u}_{\text{CS}}(t,k) =\left\langle (e^{-it(I+2L_{u_0}^{\text{CS}})}S^{*})^k  u_0, 1 \right\rangle.
\end{align}
For \eqref{eq:BO}, when $k<0$ we take $\widehat{u}(t,k) = \overline{\widehat{u}(t,-k)}$ since $u$ is real-valued, while for \eqref{eq:CS}, when $k<0$, $\widehat{u}(t,k) = 0$ since $u(t) \in L^2_{+}$.

\begin{rem}
In the case of $\mathbb{R} $, an explicit formula for \eqref{eq:BO} is also presented in Gérard's work \cite{G}, and its generalization to the full hierarchy of \eqref{eq:BO} is presented in Killip--Laurens--Vi\c san~\cite{Killip24}. For \eqref{eq:CS}, we refer to Killip--Laurens--Vi\c san \cite{KLV25} for an explicit formula on the real line $\mathbb{R}$.
\end{rem}
\medskip

In this work, instead of discretizing the underlying equations \eqref{eq:BO} and \eqref{eq:CS} as one would classically do, we obtain sharper results by exploiting the strong integrable structure of the equations encoded by the explicit formulas \eqref{eq:explicitForm}. This is the second time, after the work of \cite{ABCD24}, that this new approach is taken. The huge advantage of working with these explicit formulas instead of the underlying equations lies in the key observation that discretizing such explicit formulas yields schemes that are exact in time (requiring only spatial discretization) and have a computational cost independent of the final time 
$T$. As a result, the schemes are significantly more accurate and can be simulated up to arbitrarily large~times.

In particular, the formulas \eqref{eq:explicitForm}, written as a characterization of the $k$-th Fourier coefficient of the solution, are perfectly suited for approximating numerically, via a spectral discretization. Let $K$ be the number of Fourier frequencies used in the discretization and for $j\in \mathbb{N}$, we define the truncated Riesz--Szeg\H{o} projections $\Pi_j$ as
\begin{equation*}
\wh{ \Pi_j f}(k) = \mathbbm{1}_{0\le k <j} \wh{f}(k) \quad\text{and}\quad \Pi_\infty = \Pi .
\end{equation*}
Among the many possible discretizations of these explicit formulas, the authors of \cite{ABCD24} introduce the simplest one: by replacing, in \eqref{eq:explicitForm}, the operators $\Pi$ with their frequency cut-offs $\Pi_{K}$, they obtain a schemes which in the case of \eqref{eq:BO} reads
\[
\widehat{u_K}(t,k)=\langle(e^{it(I+2L_{K})}S^{*})^k\Pi_K u_0,1\rangle, 
\quad \text{with}\quad
L_{K}f=-i\partial_xf-\Pi_K(u_0\Pi_Kf), \quad f \in L^{2}_{+}.
\]
Using the commutation properties of the Lax pairs with the shift operator $S^{*}$, the explicit formula~\eqref{eq:explicitForm}, and an equivalence of norms, the authors then
prove convergence for sufficiently smooth solutions $u \in C([0,T], H^s)$, $s>1$, with an optimal linear dependence of the error constant $C_T$ on the final time \( T \), see also Remark \ref{rem:cv-highreg}.

In this work, we first generalize the above scheme of \cite{ABCD24} to the class of schemes \eqref{eq:newScheme}, which we will introduce in Section \ref{sec:schemes} below. 
The idea is to add degrees of freedom in the discretization: instead of truncating at each iteration $k \in \{0,...,K-1 \}$ the projections $\Pi$ in the Lax operator $L$ at frequency $K$, we introduce a sequence of truncation parameters denoted by $n_K(k)$, and allow for different truncations $L_{n_K(k)}$ depending on the iteration $k$. This additional degree of freedom allows us to uncover novel schemes that exactly preserve both the average and the $L^2$-norm of the solution on the discrete level, offering improved {\it structure-preserving} properties.

Secondly, we extend the convergence results to the above class of schemes for rough solutions $u(t)$ merely belonging to $L^2(\mathbb{T})$ or $L^2_{+}(\mathbb{T})$ for \eqref{eq:BO} and \eqref{eq:CS} respectively  (see Theorem~\ref{thm:main} below). We emphasize the fact that the tools used here for proving convergence are completely different from those of previous numerical works, and are tailored to treat rough solutions with no positive Sobolev regularity.

\begin{rem}[The cubic Szeg\H o equation]\label{rem:szego}
Gérard and Grellier \cite{gerard2015explicit} discovered an explicit formula for the cubic Szeg\H{o} equation
\begin{equation}
\partial_t u= \Pi(|u|^2u)
\label{S}\tag{S}
\end{equation}
for initial data $u_0 \in H^s_{+}$, $s\ge 1/2$.
Building upon this work, G\'erard and Pushnitski \cite{GP} used the explicit formula to continuously extend the data-to-solution map to $L^2_+$, yielding global well-posedness of the equation for $u_0 \in L^2_{+}$.

We note that a scheme for \eqref{S} was constructed in the work \cite{ABCD24} based on the explicit formula \cite{gerard2015explicit}; however, the present work only addresses the cases of \eqref{eq:BO} and \eqref{eq:CS}.  For sufficiently regular initial data $u_0 \in H^s_{+}$, $s\ge \frac12$, for which the explicit formula of Gérard and Grellier \cite{gerard2015explicit} applies, the class of numerical schemes constructed here can be adapted to the case of the \eqref{S} equation. However, our convergence analysis does not hold for \eqref{S}.

Indeed, \eqref{S} differs significantly from the other two equations \eqref{eq:BO} and \eqref{eq:CS}.  The explicit formula in \cite{gerard2015explicit,GP} is based on the Lax pair
\[
H_{u_0}(f) = \Pi(u_0 \bar{f })\quad \text{and} \quad K_{u_0}^2f=H_{u_0}^2f-\langle f, u_0\rangle u_0, \quad f \in L^2_{+}.
\]
Unlike the Lax operators \eqref{eq:lax-ops}, the operator $H_{u_0}$ does not contain a diagonal derivative term $-i\partial_x$.
 This is closely related to the lack of dispersion in \eqref{S}, as setting $u_0 = 0$ in the Lax operators appearing in the explicit formula \eqref{eq:explicitForm} yields a formula for the linear flow.
Ultimately, we will construct the operators \eqref{eq:lax-ops} for $u_0$ in $ L^2$ or $L^2_{+}$ as infinitesimally norm-bounded perturbations of $-i\partial_x$, and so this analysis will not apply to $H_{u_0}$ (see also Remark \ref{rem:perturbation}).  By comparison, the work \cite{gerard2015explicit} proved that $H_{u_0}$ is a bounded operator when $u_0\in H^s_+$ with $s\geq\frac12$.  The authors of \cite{GP} then discovered a way to construct $H_{u_0}$ as an unbounded operator for merely $u_0\in L^2_+$ via a novel approach specific to Hankel operators.
\end{rem}

\subsection{Results}

Our main result is the convergence for rough data $u_0 \in L^2(\mathbb{T})$ and $u_0 \in L^2_{+}(\mathbb{T})$, of the new class \eqref{eq:newScheme} of fully discrete schemes for \eqref{eq:BO} and \eqref{eq:CS}, as presented next.

\begin{thm}\label{thm:main}
Let $u(t)$ denote the global solution to the~\eqref{eq:BO} or~\eqref{eq:CS} equation with initial data $u_0$ in $L^2(\mathbb{T})$ or $L^2_+(\mathbb{T})$ respectively, with $\| u_0 \|_{L^2} < 1$ in the case of the focusing~\eqref{eq:CS} equation. Denote by $u_K(t)$ the numerical scheme~\eqref{eq:newScheme} corresponding to $u_0$ with $K$ frequencies and truncation parameters $\big(n_K(k)\big)_{k\geq 0}$.  If $n_K(k) \xrightarrow{K\to\infty} \infty$, for each $k\ge 0$,
then for any $T>0$ we have
\begin{equation}\label{eq:cv}
\sup_{t\in [-T, T]}\| u_{K}(t) - u(t)\|_{L^2} \xrightarrow{{ K\to\infty}
} 0.
\end{equation}
\end{thm}

A direct corollary of the above theorem is the asymptotic convergence of the $L^2$ norm of the numerical approximation to the exact solution.
\begin{cor}
Under the hypotheses of Theorem \ref{thm:main}, we have
\[
\|u_{K}(t) \|_{L^2} \xrightarrow{ K\to\infty}
\| u(t)\|_{L^2},
\]
uniformly for $t\in [-T,T]$.
\end{cor}

Before comparing our $L^2$ convergence result with those in the literature, we start by discussing the statement of Theorem \ref{thm:main} and then illustrate it through the example of the Talbot effect.

In the above theorem, we only require simple convergence of the truncation parameters $K$ and $(n_{K}(k))_{k\geq 0}$, without a prescribed order. This is the most general notion of convergence. We emphasize that it is natural to ask for the above condition on $K$ and $(n_{K}(k))_{k\geq 0}$ as we want the truncated operators to asymptotically converge to their continuous counterpart. 

We contend that the mode of convergence in our result is the strongest one can expect for $L^2$~data. First, since no additional Sobolev regularity is assumed, the best possible outcome is $L^2$-convergence as in \eqref{eq:cv}, with no quantitative rate.
Secondly, as the eigenvalues of the truncated Lax operator $L_n$ are not equal to those of $L$ for any $n\in \mathbb{N}$, the propagator $e^{itL_n}f$ will necessarily drift away from $e^{itL}f$ at sufficiently large times. Hence, bounded time intervals $[-T,T]$ are required in the analysis of the unitary groups $t\mapsto e^{itL}$ that appear in the explicit formulas \eqref{eq:explicitForm} and thereby in obtaining an $L^2$ convergence result.

The proof of our main result relies upon a careful analysis of the Lax operators $L=L_{u_0}$ and their corresponding unitary groups $t\mapsto e^{itL}$.  Specifically, the main question we need to address is: If we replace the Lax operator $L$ by a truncation $L_n$, how do we efficiently estimate the divergence of $e^{itL_{n}}f$ from $e^{itL}f$?  Our solution is inspired by the analyzes of the Lax operators in \cite{Killip24,KLV25}, and will be presented in Section~\ref{sec:pf}.
\bigskip

\noindent
{\bf Example: The Talbot effect.}
Our convergence result covers a much broader class of initial data than the work \cite{ABCD24}.  In particular, one important example included in our analysis (in the \eqref{eq:BO} case) is that of a square-wave profile, which belongs to $H^{\frac{1}{2}-}$ (but not to $L^2_+$).  In Figure~\ref{fig:talbot}, we plot the evolution of the linear and nonlinear \eqref{eq:BO} equations starting from the initial data $u_0(x)=\operatorname{sgn}(x)$ at certain rational and irrational times.

\begin{figure}[!b]
  \centering

  \begin{subfigure}{0.48\textwidth}
    \centering
    \includegraphics[width=\linewidth]{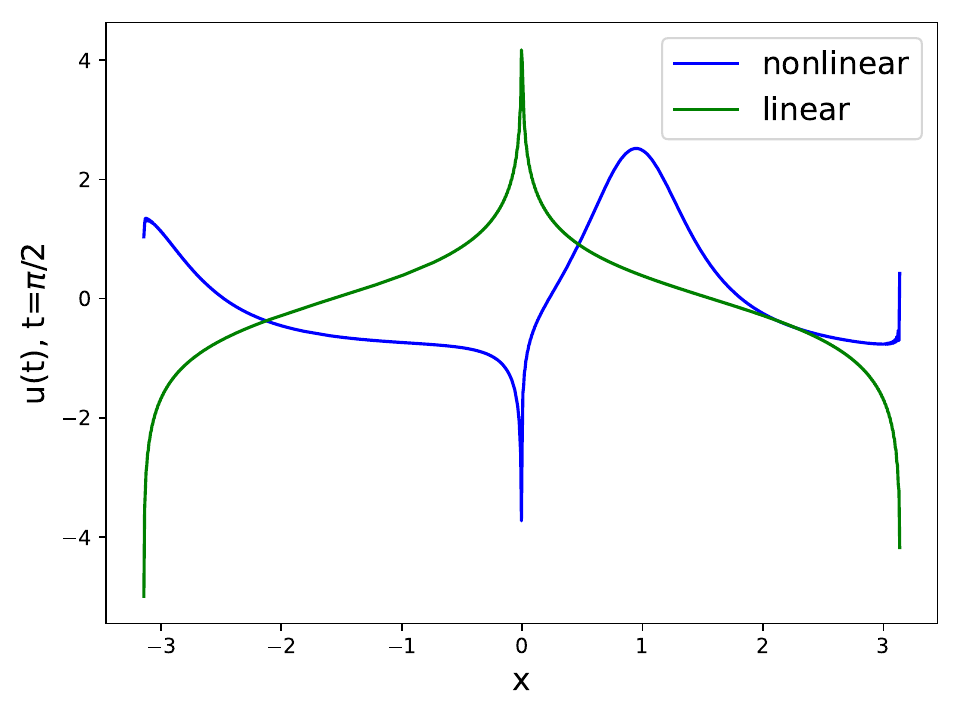}
  \end{subfigure}
  \hfill
  \begin{subfigure}{0.48\textwidth}
    \centering
    \includegraphics[width=\linewidth]{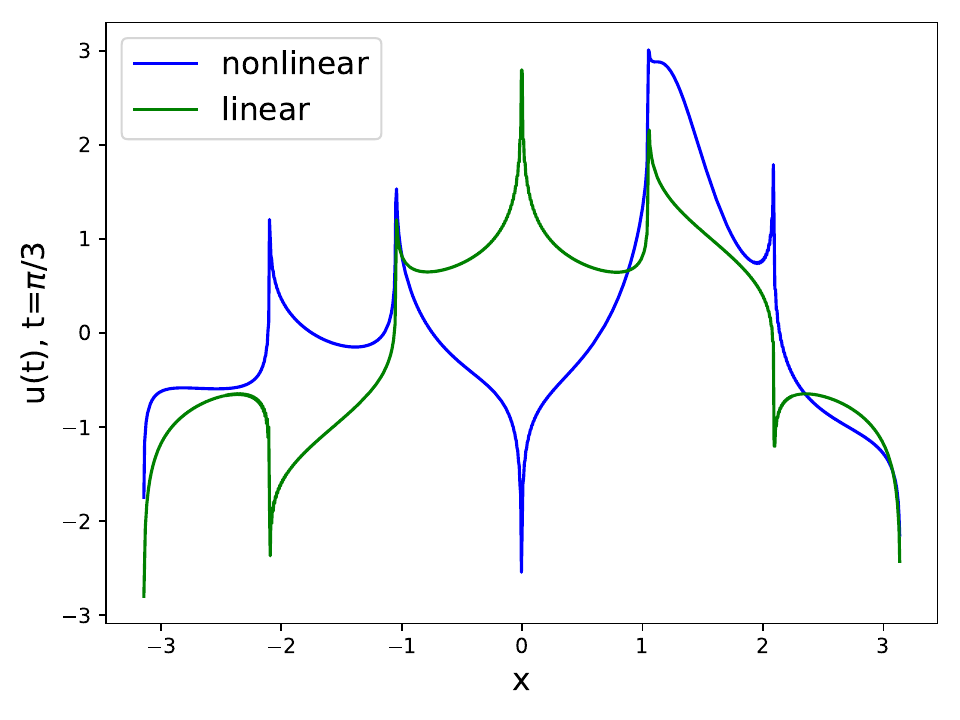}
  \end{subfigure}

  \medskip

  \begin{subfigure}{0.48\textwidth}
    \centering
    \includegraphics[width=\linewidth]{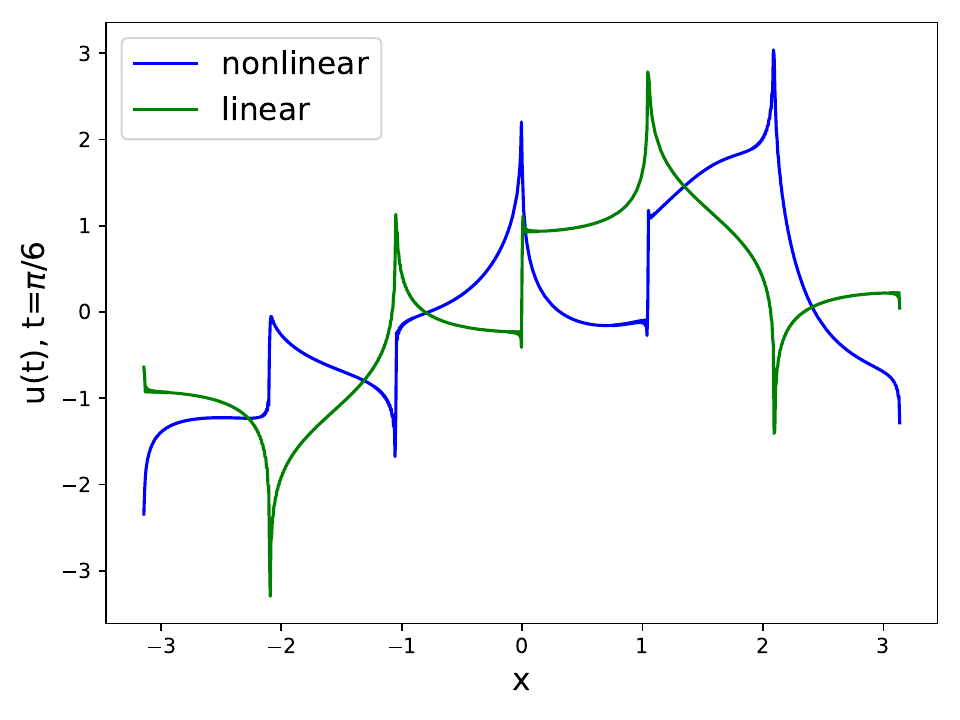}
  \end{subfigure}
  \hfill
  \begin{subfigure}{0.48\textwidth}
    \centering
    \includegraphics[width=\linewidth]{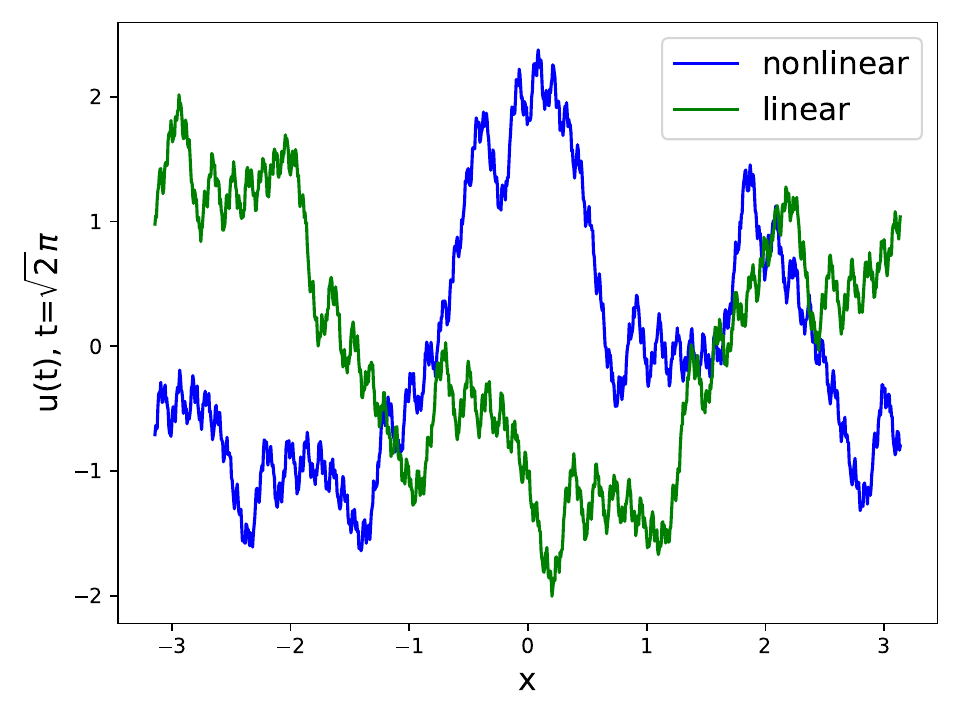}
  \end{subfigure}

  \caption{Evolution of the \eqref{eq:BO} equation from the step-function initial data
$u_0(x)=\operatorname{sgn}(x)$.
The first three panels illustrate the Talbot effect at the rational times
$t=\pi/2$, $t=\pi/3$, and $t=\pi/6$, while the last panel corresponds to the
irrational time $t=\sqrt{2}\,\pi$.
The scheme~\eqref{eq:linscheme}, shown in green,
approximates the linearized equation.
The scheme~\eqref{eq:newScheme BO}, shown in blue, uses the spectral truncation
$n(k)=K/2$ for $0\le k\le K/2$ and $n(k)=0$ otherwise and approximates \eqref{eq:BO}.
We take $K=2^{10}$.}
  \label{fig:talbot}
\end{figure}

Already, the linear dynamics $\partial_t u = \partial_x |\partial_x| u$ exhibits remarkable behavior: the profile $x\mapsto u(t,x)$ is continuous if and only if $t/2\pi$ is irrational, whereas when $t/2\pi$ is rational it possesses jump discontinuities and vertical asymptotes.  A rigorous explanation of this phenomenon was provided in \cite{boulton2021new}*{Th.~2}, which demonstrates that when $t/2\pi$ is rational, the solution is a finite linear combination of translations of the initial condition $u_0$ (which produces jump discontinuities) and its Hilbert transform $Hu_0$ (which introduces vertical asymptotes).  This is known as the Talbot effect for the linear \eqref{eq:BO} equation, named after the physicist who empirically  discovered analogous behavior in the context of the linear Schr\"odinger equation \cite{Talbot}. 

Evidently, Figure~\ref{fig:talbot} suggests that the nonlinear \eqref{eq:BO} equation exhibits a similar phenomenon.  This is corroborated by Theorem~\ref{thm:main}.  
However, a rigorous proof of the Talbot effect for the \eqref{eq:BO} equation remains an open problem.  

By comparison, it is well-known that in the case of the Schr\"odinger equation, the Talbot effect persists in the presence of a cubic nonlinearity, as was demonstrated by Erdo{\u{g}}an and Tzirakis
\cite{erdogan2014talbot}. 
Previously, these authors had established an analogous result for the Korteweg--de Vries equation in \cite{erdogan2013}.  Both results come as a consequence of a \emph{nonlinear smoothing} estimate, which establishes that the difference between the linear and nonlinear evolutions possesses more regularity than each solution does individually.  For example, in the context of the cubic nonlinear Schr\"odinger equation, both the linear and nonlinear flows lie in $H^{\frac12 -}$, the space to which the initial data belongs, whereas their difference lies in $H^{1-}$.  In particular, this implies that the difference is a continuous profile at each point in time, and so the nonlinear flow shares the same discontinuities as the linear flow.

Figure~\ref{fig:talbot} indicates that such a strategy cannot succeed for the \eqref{eq:BO} equation!  Already when $t=\frac{\pi}{2}$, we see that the difference between the linear and nonlinear evolutions is a discontinuous profile, and so does not lie in $H^{\frac12 +}$.

Nevertheless, one might hope that a nonlinear smoothing result might be restored after performing a change of variables.  This was pursued in \cites{IMOS,PKT}, which consider the gauge transformation $u\mapsto w$ of Tao \cite{Tao}, and establish a nonlinear smoothing estimate for the equation solved by $w$.  It follows that the solution $w(t,x)$ corresponding to a square-wave initial profile $w(0,x)$ exhibits a Talbot effect, but it is not immediately clear if this implies a Talbot effect in terms of the original variable $u$. We leave this question for future investigation.

Although incidental to our analysis, we note that Talbot's experiments can also be described by the Schr\"odinger equation with Dirac comb initial data, which lies in $H^{-\frac12 -}$; for details, see Banica--Vega \cite{banica2020evolution}, Eceizabarrena \cite{eceizabarrena2021talbot}, and the references therein.  However, for the \eqref{eq:BO} equation, this initial value problem is ill-posed.  Not only is this regularity outside of the regime for well-posedness, but the Dirac comb is \emph{precisely} the example that was leveraged by  Angulo Pava--Hakkaev~\cite{APH} to prove ill-posedness, which shows that the solution would have to travel infinitely fast in a certain sense.  Moreover, canceling these spatial translations by renormalizing the equation does not immediately resolve the issue, as the resulting system would still be ill-posed in $H^{-\frac12 -}$ by the example in Gérard--Kappeler--Topalov \cite{Gerard23WP}*{Appendix~B}.
\newline

\noindent
{\bf Comparison with convergence results in the literature.}
To our knowledge, this is the first error analysis result that establishes convergence of a numerical scheme in \( C([0,T], L^2(\mathbb{T})) \) norm for merely \(L^2\) initial data, and for {\it arbitrary} times \(T>0\). 

For~\eqref{eq:CS},  the only convergence analysis available is provided in~\cite{ABCD24}, which applies only to sufficiently smooth solutions in $H^s_+$ with $s>1$, see also Remark \ref{rem:cv-highreg}. 

For \eqref{eq:BO}, existing convergence results for rough data \(u_0 \in L^2\) have been obtained only in the weaker norm \( L^2([0,T], L^2_{\mathrm{loc}}(\mathbb{R})) \), for {\it some} finite \(T>0\), and are restricted to the real line \( \mathbb{R} \).
Namely, on \( \mathbb{R} \), by combining Kato smoothing effect with Aubin--Lions compactness, Galtung \cite{galtung-18} proved convergence of a Crank--Nicolson Galerkin scheme for \eqref{eq:BO}  in the space \( L^2([0,T], L^2_{\mathrm{loc}}(\mathbb{R})) \), for \(u_0 \in L^2(\mathbb{R})\)  and some \(T>0\). This result builds upon the earlier error analyses of \cite{dutta2016note, dutta2015convergence}, where the same convergence was established for the Korteweg--de Vries (KdV) equation using similar arguments based on Kato smoothing; see also \cite{amorim2013convergence, holden2015convergence} for earlier works of the same kind.
On the torus \( \mathbb{T} \), such local smoothing effects do not hold, and therefore a different approach is required. For the Benjamin--Ono equation, this was considered in \cite{dutta-16}, but only for regular data. In that work, convergence as \( \Delta x \to 0\) and \( \Delta t = O(\Delta x)\) of a Crank--Nicolson scheme was shown in the more regular norm \( C([0,T] \times \mathbb{T}) \), for some \(T>0\) and \(u_0 \in H^2(\mathbb{T})\).

Finally, we note that in recent years, several works have introduced refined error analysis techniques based on discrete Bourgain spaces or discrete Strichartz estimates to establish convergence of numerical schemes for semilinear dispersive equations with low-regularity data. These approaches provide explicit convergence rates in \( C([0,T], L^2) \) norm in terms of the Sobolev regularity of the initial data $u_0 \in H^s$, $s>0$, for any final time $T$ below the maximal time of existence. However, they do not cover the case of \(L^2\)-solutions (\(s=0\)); see for example \cite{rousset-22} for the KdV equation and \cite[Chapter 1, Section 3]{AB-2025book} for an overview of related results. 
To enable a uniform $L^2$-convergence result for $L^2$ data we introduce completely different techniques based on compactness arguments and Lax operators, together with these new explicit formulas.

\begin{rem}[Convergence of $u_K$ for higher regularity data]\label{rem:cv-highreg}
Let $u_K$ be the scheme defined by \eqref{eq:newScheme} with $n(k) \ge cK$, for some constant $c>0$ independent of $K$. By applying the same proof as in~\cite{ABCD24} with truncation parameter $\tilde K := cK$ we obtain their convergence result for this subclass $u_K$ of schemes, namely that for any $t\in \mathbb{R}$, $0\le r\le s$ and $u_0 \in H^s(\mathbb{T})$ with $s>1$, there exists $C>0$ such that
\[
\|u_K(t) - u(t)\|_{H^r} \le C(1+t)K^{-s+1+r}.
\]
We note that the regularity condition $s>1$ is natural, and is due to the structure of the explicit formula \eqref{eq:explicitForm}. Indeed, discretizing each exponential $e^{2itL}$ 
induces an $L^2$ error of order $K^{-s}$, yielding the rate $K^{-s+1}$ when applied $\tilde K$ time to control the error terms $(e^{2itL}S^{*})^k - (e^{2itL_{\tilde K}}S^{*})^k$ over ${0\le k <\tilde K}$.
As the error is measured either in $L^2$ or in a stronger $H^r$ norm, this explains why we need regularity at least $s>1$ to prove convergence rates.

We note that the schemes satisfying condition \eqref{eq:m-L2}, which preserve the $L^2$-norm on the discrete level, do not fit in the above case. Thereby, the only convergence result available for these schemes is the one presented here.
\end{rem}

\subsection{Outline}
Section~\ref{sec:schemes} is devoted to the introduction of our class of numerical schemes. Specifically, in Section~\ref{sec:defschemes} we define the schemes based on the explicit formulas presented in the introduction. We then discuss their preservation properties in Section~\ref{sec:struc}, and their computational cost in Section~\ref{sec:CPU}. In Section~\ref{sec:notation} we introduce the preliminary material needed for the proof of our main theorem, which is given in Section~\ref{sec:pf}.

\subsection*{Acknowledgements} The work of Y.A.B. is funded by the National Science Foundation through the award
DMS-2401858. The work of T.L.\ was supported by NSF CAREER grant DMS-1845037 and an AMS-Simons Travel Grant.  The authors thank  Louise Gassot and the Institut de recherche mathématique de Rennes (IRMAR) for support and hospitality at an early stage of this project, as well as Andreia Chapouto for discussions about the Talbot effect. 

\section{Schemes based on the explicit formulas}\label{sec:schemes}
\subsection{Defining the class of schemes}\label{sec:defschemes} 
The general form for the scheme $u_K$ is given in Fourier space as follows.
Let $K\in \mathbb{N}$ be the number of Fourier frequencies used in the discretization.
For each $K$, choose a sequence of truncation parameters $n(k) = n_K(k) \in \mathbb{N}$, for $k\in\mathbb{N}$, and set

\begin{equation}\label{eq:newScheme}
\widehat{u_{K}}(t,k) = \langle u^k, 1\rangle \quad\text{for}\quad 0\le k < K ,
\end{equation}
where
$u^0 = \Pi_{n(0)} u_0$, and for $k \ge 1$
\begin{equation}\label{eq:newScheme BO}
u^{k} = \left(e^{it(I + 2 L_{n(k)})}S^{*}\right)u^{k-1} \ \text{with}  \quad L_{j} = -i\partial_x - \Pi_{j}u_0\Pi_{j}, \quad j\in\mathbb{N}
\end{equation}
in the case of \eqref{eq:BO}, while for \eqref{eq:CS} 
\begin{equation}\label{eq:newScheme CS}
u^{k}  = \left(e^{-it(I + 2L_{n(k)})}S^{*}\right)u^{k-1}\  \text{with} \quad L_{j} = -i\partial_x \pm \Pi_{j}u_0\Pi_{j}\ol u_0 \Pi_{j}, \quad j\in\mathbb{N}.
\end{equation}
In the case of~\eqref{eq:CS}, by setting $\widehat{u_K}(t,k) = 0$ for $k<0$ or $k \ge K$, the formula~\eqref{eq:newScheme} defines an element of $L^2_+$, yielding our approximation of $u(t)$.  For~\eqref{eq:BO}, in order to ensure that $u_K$ is real-valued we set
\begin{equation*}
\widehat{u_{K}}(t,k) = \overline{\widehat{u_K}(t,-k)} \quad\text{for}\quad {-K} < k <0, 
\end{equation*}
and then take $\widehat{u_K}(t,k) = 0$ for $|k|\ge K$. 
\begin{rem}
The scheme \eqref{eq:newScheme} is independent of the values of $n(k)$ for $k \ge K$, as the iterates $u^k$ are only computed for $0\le k <K$. Nevertheless, we define the scheme with the full sequence $(n(k))_{k\ge 0}$ as it allows us to identify those which have better structure preserving properties, as further detailed below.
\end{rem}
\begin{rem}
The matrix exponentials appearing in \eqref{eq:newScheme BO} and \eqref{eq:newScheme CS} are well defined for $u_0\in L^{2}$ and $L^2_{+}$ respectively. Indeed, we show in Section \ref{sec:adj-equiv} that for any $j \in \mathbb{N}$ the Lax operators $L_j$ appearing in \eqref{eq:newScheme BO} and \eqref{eq:newScheme CS} are semi-bounded self-adjoint operators on $H^1_{+}$ and hence generate a strongly continuous one-parameter unitary group $t \to e^{itL_{j}}$ on $L^2_{+}$.
\end{rem}

\subsection{Structure-preserving properties and examples}\label{sec:struc}
Before giving examples of schemes belonging to this class \eqref{eq:newScheme}, we discuss some preservation properties which the class inherits.
By considering the coefficient $\widehat{u_{K}}(t,0)$, we have that the schemes \eqref{eq:newScheme} preserve the average of $u_0$.  This is a conserved quantity for both~\eqref{eq:BO} and~\eqref{eq:CS}; in the case of \eqref{eq:BO} this corresponds to the  mass.
\begin{prop}\label{prop:average}
Let $n(0)\geq 1$.  Then for all $t\in \mathbb{R}$ the scheme $u_K$ in \eqref{eq:newScheme} satisfies
\[\int_{\mathbb{T}} u_K(t,x)\,dx = \int_{\mathbb{T}} u_0(x)\,dx.\]
\end{prop}

The $L^2$ norm is another conserved quantity that is common to both systems.
A second important feature of the above schemes \eqref{eq:newScheme} is that they have $L^2$ norm bounded by that of the initial data, as expressed in the following proposition.

\begin{prop}\label{prop:L2 consv}
Given $t \in \mathbb{R}$ we have 
\begin{equation}
\|\Pi u_{K}(t)\|_{L^2} \le  \| \Pi_{n(0)} u_0\|_{L^2}.
\label{L2 consv}
\end{equation}
In particular, this implies that 
\begin{equation}\label{eq:l2-est}
\| u_K(t)\|_{L^2} \le \| u_0\|_{L^2}.
\end{equation}
\end{prop}
\begin{proof}
Fix the initial data $u_0\in L^2$ or $u_0 \in L^2_{+}$, and let $u_K$ denote the numerical scheme constructed in \eqref{eq:newScheme} from the recursive sequence $u^k$ defined in \eqref{eq:newScheme BO} for \eqref{eq:BO} and in \eqref{eq:newScheme CS} for \eqref{eq:CS}.

As each $L_{n(k)}$ is self-adjoint (see Propositions \ref{prop:lax} and \ref{prop:lax'} below for details), it generates a one-parameter unitary group $t\mapsto e^{itL_{n(k)}}$.  Consequently, for $k\geq 1$ we have
\begin{align*}
\| u^k \|_{L^2}^2 
= \| e^{\pm it(I+2L_{n(k)})} S^* u^{k-1} \|_{L^2}^2 
= \| S^* u^{k-1} \|_{L^2}^2 
= \| u^{k-1} \|_{L^2}^2 - \big| \langle u^{k-1}, 1 \rangle \big|^2 .
\end{align*}
By induction, we conclude
\begin{equation}\label{eq:ineq-L2}
\begin{aligned}
0\le \| u^{K} \|_{L^2}^2 
= \| u^{0} \|_{L^2}^2 - \sum_{k=0}^{K-1} \big| \langle u^{k}, 1 \rangle \big|^2 
= \| \Pi_{n(0)}u_0 \|_{L^2}^2 - \| \Pi u_K(t) \|_{L^2}^2.
\end{aligned}
\end{equation}
In particular, using Proposition \ref{prop:average} we have
\begin{equation}\label{eq:l2bd}
\| u_K(t)\|_{L^2}^2 = \| \Pi u_K(t)\|_{L^2}^2 + \big\| \overline{\Pi u_K(t)}\big\|_{L^2}^2 - \langle u_K(t), 1\rangle^2 
\le 2\| \Pi_{n(0)}u_0\|_{L^2}^2 - \langle u_0, 1\rangle^2 \le \| u_0\|_{L^2}^2.
\end{equation}
\qedhere
\end{proof}
The $L^2$-inequality~\eqref{L2 consv} is a vital step necessary for going from weak to strong $L^2$~convergence, and is hence essential for the convergence analysis presented in Section \ref{sec:pf}.
\newline

We now give examples of various schemes belonging to the above class \eqref{eq:newScheme}.

\begin{ex}[The linear case]
The simplest example consists in taking $n(0)=K$ and $n(k) = 0$, for  $k\ge 1$. The scheme thereby obtained is given by
\begin{equation}\label{eq:linscheme}
\widehat{u_{K}}(t,k) = \langle (e^{\alpha it(I-2i\partial_x)}S^*)^k\Pi_{K}u_0, 1\rangle, \quad 0\le k < K,
\end{equation}
with $\alpha = 1$ for \eqref{eq:BO} and $\alpha = -1$ for \eqref{eq:CS}.
The above scheme approximates the linearized equation
$
\partial_t u = \alpha \partial_x |\partial_x| u.
$
Indeed, it is straightforward to verify that for $K=\infty$ the above formula is equivalent to the linear flow $e^{\alpha t\partial_x|\partial_x|} u_0$.
\end{ex}

\begin{ex}[The scheme of \cite{ABCD24}] By taking $n(k)= K$, for each $k\ge 0$, we recover the scheme recently obtained by Alama Bronsard--Chen--Dolbeault \cite{ABCD24}.
\end{ex}
Next, by taking a more subtle approximation of the Lax operator, and hence of $n(k)$, we can obtain schemes which preserve the $L^2$-norm of the truncated initial data.
\begin{ex}[Schemes preserving the discrete $L^2$-norm]
By taking 
\begin{equation}\label{eq:m-L2}
n(k) \le K-k \quad \text{for} \quad 0\le k< K
\end{equation}
we obtain new schemes which {\it exactly preserve} the $L^2$-norm on the discrete level. This is expressed in the following proposition. 
\end{ex}
\begin{prop}\label{prop:l2norm-pres}
Let $t\in\mathbb{R}$, and let $u_K$ be any scheme defined by \eqref{eq:newScheme} satisfying assumption~\eqref{eq:m-L2}. This scheme satisfies
\[
\|\Pi u_{K}(t)\|_{L^2} =  \| \Pi_{n(0)} u_0\|_{L^2}
\]
\end{prop}
\begin{proof}
The proof follows from that of Proposition \ref{prop:L2 consv}, by noticing that $u^K \equiv 0$ and hence equality holds in \eqref{eq:ineq-L2}.  Indeed, in frequency variables, each $u^k$ is of size $K-k$ (see Remark \ref{rem:size-uk}) and is supported between the frequencies $0$ and $K-k-1$, and so in particular $u^K$ must vanish everywhere. Hence, it follows that
\[
\| \Pi u_K(t)\|_{L^2}^2 = 2 \| \Pi_{n(0)}u_0\|_{L^2}^2 - \langle u_0, 1 \rangle^2 = \| \Pi_{n(0)} u_0\|_{L^2}^2.
\]
\end{proof}
\begin{rem}
We cannot expect equality in \eqref{eq:l2-est}, since the highest modes of $u_0$ have been truncated. Consequently, the second estimate in \eqref{eq:l2bd} is necessarily an inequality. The strongest result one can hope for is that the $L^2$-norm of the truncated initial data is conserved in time, which is exactly what Proposition \ref{prop:l2norm-pres} guarantees.
\end{rem}

\begin{rem}[The size of the iterates $u^k$]\label{rem:size-uk}
In frequency variables,
the iterates $u^k$ in the numerical schemes \eqref{eq:newScheme BO} and \eqref{eq:newScheme CS} are
 vectors
of size
\be\label{eq:max-size}
m_k := \max_{0\le \ell \le k}\{ n(\ell)-(k-\ell)\},
\ee
which is bounded for $0\le k <K$ by $M = \max_{0\le k < K}n(k)$. 
Indeed, $u^{0}$ is of size $n(0)$; after shifting, $S^{*}u^0$ is of size $n(0)-1$; and when applying $e^{itL_{n(1)}}$, $u^1$ is of size $\max\{ n(0)-1, n(1)\}$, since the diagonal part $-i\partial_x$ is not truncated by $\Pi_j$ in \eqref{eq:newScheme BO} and \eqref{eq:newScheme CS}. Thereby, \eqref{eq:max-size} follows by induction.
\end{rem}

\begin{rem}[Truncated Lax operators and perturbations of $-i\partial_x$]\label{rem:perturbation}
As discussed in Remark~\ref{rem:size-uk}, the scheme is implemented as a vector of size $M$, and the computation of $L_n$ is equal to that of $\Pi_M L_n \Pi_M$. Nevertheless, for the analysis we do not apply any truncation $\Pi_M$ on the first term $-i\partial_x$ in the definition of the truncated Lax operators \eqref{eq:newScheme BO} and \eqref{eq:newScheme CS}. Indeed, this is crucial in order to treat the Toeplitz term $\Pi_nu\Pi_n$ or $\Pi_{n}u\Pi_{n}\ol u \Pi_n $ as perturbations of $-i\partial_x$, and thereby to obtain the desired convergence result, see Section \ref{sec:adj-equiv}. For this same reason, our convergence proof does not apply to the cubic Szeg\H o equation: its Lax operator lacks the derivative term $-i\partial_x$, so the problem cannot be treated perturbatively.
\end{rem}

\subsection{Computational cost}\label{sec:CPU}
The leading cost comes from computing the matrix exponentials in \eqref{eq:newScheme BO} and \eqref{eq:newScheme CS}, approximating the flow of the Lax operator $e^{2itL}$.
Hence, the cost of the scheme \eqref{eq:newScheme} will necessarily depend on the choice of the truncation of the Lax operator, which is dictated by $n(k)$. 

If no structure is imposed, and the $n(k)$ are chosen arbitrarily, then the scheme is costly to implement and requires $O(K^4)$ operations. This is due to the fact that in order to compute the Fourier coefficients of the scheme \eqref{eq:newScheme}, we compute $K-1$ matrix exponentials of the form $e^{it(I + 2L_{n(k)})}$, which each cost $\mathcal O(K^3)$.

However, if structure is imposed, such as when taking $n(k) = K$, or the $L^2$-preserving scheme corresponding to $n(k) = K/2$ for $0\le k\le K/2$ and zero otherwise (which satisfies \eqref{eq:m-L2}),
then the matrix exponential  $e^{it(I + 2L_{n(k)})}$ can be computed once in $ \mathcal O(K^3)$.
Indeed, 
since the matrices are self-adjoint, they can be diagonalized as $P\Lambda P^T$, allowing us to compute $Pe^{it\Lambda}P^T$ in $\mathcal O(K^3)$ operations. Moreover, as the eigenvalues $\lambda_j$ of the diagonal matrix $\Lambda$ are real-valued, we have $|e^{it\lambda_j}| = 1$ for all $t$, ensuring the stability of the method over time.
Once the matrix exponentials are computed, we can evaluate all the vectors in $\mathcal O(K)$ matrix-vector multiplications, with a computational cost in $\mathcal O(K^3)$ once again.

We emphasize that what counts is the {\it balance} between the precision $\epsilon$ and the computational cost $\mathcal{C}$ of the method, and not solely one of these quantities. When considering the trade-off between $\epsilon$ and $\mathcal{C}$, the schemes based upon explicit formulas outperform those previously established in the literature, as discussed in \cite[Section 4.2]{ABCD24}.

\section{Fourier transforms, norms and resolvants}\label{sec:notation} 

Our convention for the inner product on $L^2$ is
\[
\langle f, g \rangle = \frac{1}{2\pi} \int_{-\pi}^{\pi} f\ol{g}\,dx.
\]
We then choose to define the Fourier transform as
\begin{equation*}
\wh{f}(k) = \langle f, e^{ikx} \rangle, \quad\text{so that}\quad f(x) = \sum_{k\in\Z} \wh{f}(k) e^{ikx}.
\end{equation*}
Plancherel's identity then reads
\begin{equation*}
\|f\|_{L^2}^2 = \sum_{k\in\Z} |\wh{f}(k)|^2.
\end{equation*}
We define the operator norm as follows
\[
\| A\|_{\op} := \| A\|_{L^2_{+} \to L^2_{+}} = \sup_{\substack{\| f\|_{L^2} \le 1 \\ f\in L^2_{+}}} \| Af\|_{L^2}.
\]

We now recall the truncated Lax operators for \eqref{eq:BO} and \eqref{eq:CS} on $L^2_+$, which we {\it formally} defined in \eqref{eq:newScheme BO} and \eqref{eq:newScheme CS} when writing the scheme.
For \eqref{eq:BO}, we let $u\in L^2$ be real-valued initial data,  and define
\begin{equation}\label{eq:BO-trunc}
L_n = -i\partial_x - \Pi_n u \Pi_n \quad\text{and}\quad L_\infty = L = -i\partial_x - \Pi u .
\end{equation}
For \eqref{eq:CS}, we let $u\in L^2_+$ and construct the truncated Lax operators
\begin{equation}\label{eq:CS-trunc}
L_n = -i\partial_x \mp \Pi_n u \Pi_n\ol{u} \Pi_n \quad\text{and}\quad L_\infty = L = -i\partial_x \mp u\Pi \ol{u} .
\end{equation}
In the next section, we will construct $L_n$ as a perturbation of the free operator $L_0 = -i\partial_x$ on $L^2_+$, with domain $D(L_0) = H^1_{+}$.  We denote the resolvent of the latter by
\begin{equation*}
R_0(\kappa) = (L_0 + \kappa)^{-1}, \quad \kappa \ge 1.
\end{equation*}
It will also be convenient to define norms adapted to this resolvent:
\begin{equation}\label{eq:Hskappa}
\|f\|_{H^s_\kappa}^2 = \sum_{k\in\Z} (|k|+\kappa)^{2s} |\wh{f}(k)|^2.
\end{equation}


\section{Proof of convergence}\label{sec:pf}

The goal of this section is to prove our main convergence result presented in Theorem \ref{thm:main}. The proof is given in Section \ref{sec:pfThms}, where we show how strong convergence in $L^2(\mathbb{T})$ follows from combining weak convergence and Proposition~\ref{prop:L2 consv}.
To prove weak convergence we rely on the result of Section~\ref{sec:cv-group}, which proves convergence of the unitary group $e^{itL_n}$ associated to the truncated Lax operator $L_n$, uniformly on bounded time intervals $[-T,T]$ and over compact subsets $\mathcal{F}$ of~$ L^2_{+}$. To prove the latter, we show in Section \ref{sec:norm-res-cv} norm resolvent convergence of $L_n$, which in turn uses the operator bounds and equivalence of norms of Sections \ref{sec:op-bds} and \ref{sec:adj-equiv}.

\subsection{Operator bounds}\label{sec:op-bds}

The following estimates will be used in the next sections to rigorously define the truncated Lax operators $L_n$ in \eqref{eq:BO-trunc} and \eqref{eq:CS-trunc}, construct their resolvents, and study their mapping properties.  Our first two lemmas show that $\Pi_n u \Pi_n$ and $\Pi_n u \Pi_n \overline{u} \Pi_n $ are relatively bounded perturbations of $L_0$.  In fact, these estimates are even more general, as they do not require the left-most truncation $\Pi_n$, and just consider $u\Pi_n u$ and $u\Pi_n \bar{u} \Pi_n$, see also  Remark~\ref{rem:pert}. 

We start with the case of the \eqref{eq:BO} equation.
\begin{lem}
For any $u\in L^2$, we have
\begin{align}
\| u\Pi_n R_0(\kappa) \|_{\op} &\le\sqrt3 \kappa^{-\frac12} \|u\|_{L^2},
\label{est 1}
\end{align}
uniformly for $1\leq n\leq \infty$ and $\kappa\geq 1$.
\end{lem}
\begin{proof}
For any $s>\frac12$, the space $H^s$ embeds into $L^\infty$.  More precisely, for $g \in H^s_\kappa$, using Cauchy--Schwarz in Fourier variables we find
\begin{equation}
\|g\|_{L^\infty} \leq \|\wh{g}\|_{\ell^1} \leq \sqrt{\frac{2s +1}{2s-1}} \kappa^{\frac12 - s} \|g\|_{H^s_\kappa}, \quad \kappa\geq 1.
\label{est 1p1}
\end{equation}
Taking $s=1$ in the above and recalling the $H^s_{\kappa}$-norm \eqref{eq:Hskappa}, we see that for all $1\leq n\leq \infty$ and $\kappa\geq 1$,
\begin{equation*}
\|u\Pi_n R_0f \|_{L^2} 
\leq \|u\|_{L^2} \|\Pi_n R_0f \|_{L^\infty}
\le\sqrt3 \kappa^{-\frac12} \|u\|_{L^2} \|R_0f\|_{H^1_\kappa}
= \sqrt3\kappa^{-\frac12} \|u\|_{L^2} \|f\|_{L^2} .
\end{equation*}
This proves~\eqref{est 1}.
\end{proof}


We now show the analogue of the above lemma for the \eqref{eq:CS} equation. 
Note that the Lax operator now contains two occurrences of $u$, so the elementary bound \eqref{est 1p1} is no longer sufficient to control both factors. Instead, we leverage the Fourier support of $u\in L^2_+$.  Moreover, the operator bound in Lemma \ref{lem:CS} does not provide an explicit decay rate in $\kappa$; therefore, a separate argument is required in the case of \eqref{eq:CS} to establish decay as $\kappa \to \infty$.

\begin{lem}\label{lem:CS}
For any $u\in L^2_+$, $1\leq n\leq \infty$, and $\kappa\geq 1$, we have
\begin{align}
\| u \Pi_n \ol{u}\Pi_n R_0(\kappa)\|_{\op} &\le 2 \| u\|_{L^2}^2 .
\label{est 4}
\end{align}
Moreover, for each $u\in L^2_+$,
the left-hand side tends to 0 as $\kappa \to \infty$, uniformly for $1\le n\le\infty$.
\end{lem}
\begin{proof}

First, we have
\begin{align*}
\| u \Pi_n \ol u \Pi_n R_0 f\|_{L^2} \le \| u\|_{L^2} \| \Pi_n \ol u \Pi_n R_0 f\|_{L^\infty}.
\end{align*}
We bound the $L^\infty$ norm by the $\ell^1$ norm of the Fourier coefficients:
\begin{align*}
\| \Pi_n \ol u \Pi_n R_0 f\|_{L^\infty} &\le \sum_{k=0}^{n-1} \Bigg|\sum_{l=0}^{n-1} \ol{\wh{u}(l-k)} \frac{\wh{f}(l)}{l+\kappa}\Bigg|\\
&\le  \sum_{l=0}^{n-1}  \frac{|\wh{f}(l)|}{l+\kappa} \sum_{k=0}^{l}|{\wh{u}(l-k)}|\\
&\le \| f\|_{L^2} \Bigg\| \frac{1}{l+\kappa}\sum_{m=0}^{l}|\wh u(m)|\Bigg\|_{\ell^2} ,
\end{align*}
where in the second line we use the fact that $u\in L^2_{+}$, and hence the terms are zero for $l< k$. As $\kappa \ge 1$ we apply Hardy's inequality to obtain the bound
\[
\Bigg\| \frac{1}{l+\kappa}\sum_{m=0}^{l}|\wh u(m)|\Bigg\|_{\ell^2} \le 
\Bigg\| \frac{1}{l+1}\sum_{m=0}^{l}|\wh u(m)|\Bigg\|_{\ell^2} \le 2 \| u\|_{L^2},
\]
which proves \eqref{est 4}.

Moreover, it follows from the above that
\[
\Bigg\| \frac{1}{l+\kappa}\sum_{m=0}^{l}|\wh u(m)|\Bigg\|_{\ell^2_{l\ge N}} \xrightarrow{N \rightarrow \infty} 0.
\]
So, for $N\geq 1$ fixed, we estimate the contribution of the low frequencies $l < N$:
\begin{align*}
\Bigg\| \frac{1}{l+\kappa}\sum_{m=0}^{l}|\wh u(m)|\Bigg\|_{\ell^2_{l< N}}^2 &\le \frac{1}{\kappa^2} \sum_{l=0}^{N-1} \left(\sum_{m=0}^l  |\wh u(m)| \right)^2\\
& \le \frac{1}{\kappa^2} \sum_{l=0}^{N-1} (l+1)\sum_{m=0}^l |\wh u(m)| ^2 \\
&\le \frac{N^2}{\kappa^2} \| u\|_{L^2}^2 \xrightarrow{\kappa \to \infty} 0.
\end{align*}
Therefore, the desired convergence follows from first choosing $N$ large and then sending $\kappa\to\infty$.
\end{proof}


The following lemma is used in Section \ref{sec:norm-res-cv} for obtaining resolvent convergence as $n\to \infty$.
\begin{lem}
We have
\begin{equation}\label{eq:decay}
\| (\Pi-\Pi_n)R_0\|_{\op} = \| R_0(\Pi-\Pi_n)\|_{\op} \le \frac{1}{n},
\end{equation}
uniformly for  $\kappa\geq 1$.
\end{lem}
\begin{proof}
First, note that $R_0$ and $\Pi-\Pi_n$ commute as they are both Fourier multipliers. For $f\in L^2_{+} $, we have
\[
\| R_0(\Pi-\Pi_n)f\|_{L^2}^2 = \sum_{k\ge n}\frac{|\wh f(k)|^2}{(k + \kappa)^2} \le \frac{1}{n^2} \| f\|_{L^2}^2 .
\qedhere
\]
\end{proof}

\begin{rem}\label{rem:pert}
As $\| \Pi_n \|_{\op} \le 1$, the estimates \eqref{est 1} and \eqref{est 4} also hold when adding the truncation operator $\Pi_n$ in front and estimating the perturbative term $\Pi_nu\Pi_n R_0$ or $\Pi_n u\Pi_n \overline{u}R_0$. Nevertheless, we keep the above more general statement, given that in Section \ref{sec:norm-res-cv} we apply equation \eqref{eq:decay} followed by equations \eqref{est 1} and \eqref{est 4}, without truncation  $\Pi_n$ on the left side.
\end{rem}

\subsection{Self-adjoint operators and equivalence of norms}\label{sec:adj-equiv}

\begin{prop}[The truncated Lax operators for \eqref{eq:BO}]
\label{prop:lax}
Given $u\in L^2$ real-valued and $1 \le n \le \infty$, the operator
\begin{equation*}
L_n  = -i\partial_x - \Pi_n  u \Pi_n 
\end{equation*}
with domain $D(L_n) = H^1_+$ is self-adjoint and semi-bounded.  Moreover, there is a constant $\kappa_0 = \kappa_0(\|u\|_{L^2}) \geq 1$ such that for all $\kappa\geq\kappa_0$, the resolvent of $L_n$,
\[
R_n(\kappa;u) = (L_n + \kappa)^{-1}
\] 
exists, and maps $L^2_+$ into $H^{1}_+$. Moreover, the following estimates hold 
\begin{align}
\tfrac12\|f\|_{H^1_\kappa} &\le \| (L_n+\kappa)f \|_{L^2} \le \tfrac32 \|f\|_{H^1_\kappa} ,
\label{L} \\
\tfrac23\| f\|_{H^{-1}_\kappa} &\le\;\; \| R_n(\kappa)f \|_{L^2}\;\; \le2 \| f\|_{H^{-1}_\kappa}
\label{R}
\end{align}
uniformly for $1\leq n\leq \infty$ and $\kappa\geq\kappa_0$.
\end{prop}

\begin{proof}  First, the linear operator $L_0 = -i\partial_x$ is self-adjoint with domain $H^1_{+}$. Hence, we consider the  symmetric operator $\Pi_n u \Pi_n $.
For $f\in H^1_+$, by applying estimate \eqref{est 1} we have
\begin{equation}
\| \Pi_nu\Pi_n f \|_{L^2}
\le \| \Pi_n\|_\op \| u\Pi_n R_0\|_\op \| (L_0+\kappa)f \|_{L^2} 
\le \sqrt3 \kappa^{-\frac12} \|u\|_{L^2} \| (L_0+\kappa)f \|_{L^2}.
\label{est 1p2}
\end{equation}
By taking $\kappa$ large, it follows from the above estimate that the operator $\Pi_nu\Pi_n$ is infinitesimally small with respect to the operator $L_0 = -i\partial_x$ on $L^2_+$ \cite[Chp.~X.2]{ReedSimonvol2}.  
In particular, by choosing $\kappa_0 \ge \max{\{12\| u\|_{L^2}^2,1\}}$ we have that
\begin{equation}\label{eq:BO-bdd}
\| \Pi_nu\Pi_n f \|_{L^2}
\leq \tfrac{1}{2} \| (L_0+\kappa)f\|_{L^2}, \quad \kappa \ge \kappa_0, \quad 1\leq n\leq \infty
\end{equation}
and hence $\Pi_nu\Pi_n$ is $L_0$-bounded.
Thus, by the Kato--Rellich theorem \cite{ReedSimonvol2}*{Th.~X.12}, we have that $L_n$ is a self-adjoint operator with domain $H^1_+$, and is semi-bounded from below by $-\kappa_0$. 
Finally, by again applying \eqref{eq:BO-bdd} we have
\begin{equation*}
\tfrac{1}{2} \| (L_0+\kappa)f\|_{L^2}
\leq \| (L_n+\kappa)f\|_{L^2}
\leq \tfrac{3}{2} \| (L_0+\kappa)f\|_{L^2} ,
\end{equation*}
which proves~\eqref{L}.  The estimate \eqref{R} then follows by duality.
\end{proof}

\begin{prop}[The truncated Lax operators for \eqref{eq:CS}]
\label{prop:lax'}
Given $u\in L^2_+$ and $1\leq n\leq \infty$, the operator
\begin{equation*}
L_n f = -i\partial f \mp \Pi_n [ u \Pi_n (\ol{u} \Pi_nf) ]
\end{equation*}
with domain $D(L_n) = H^1_+$ is self-adjoint and semi-bounded.  Moreover, there is a constant $\kappa_0 = \kappa_0(u) \geq 1$ so that for all $\kappa\geq\kappa_0$, the resolvent $R_n(\kappa;u)$ of $L_n$ exists, maps $L^2_+$ into $H^{1}_+$, and 
\begin{align}
\tfrac12\|f\|_{H^1_\kappa} &\le \| (L_n+\kappa)f \|_{L^2} \le \tfrac32 \|f\|_{H^1_\kappa} ,
\label{L'} \\
\tfrac23\| f\|_{H^{-1}_\kappa} &\le\;\; \| R_n(\kappa)f \|_{L^2}\;\; \le2 \| f\|_{H^{-1}_\kappa}
\label{R'}
\end{align}
uniformly for $1\leq n\leq \infty$.
\end{prop}
\begin{proof}
Thanks to the convergence result of Lemma \ref{lem:CS}, there exists $\kappa_0 = \kappa_0(u) \geq 1$ such that for all $\kappa\geq\kappa_0$ we have
\begin{equation*}
\| \Pi_nu\Pi_n\ol{u}\Pi_n f \|_{L^2}
\leq \| \Pi_n \|_{\op} \| u\Pi_n\ol{u}\Pi_n R_0 \|_{\op} \| (L_0+\kappa)f\|_{L^2}
\leq \tfrac{1}{2} \| (L_0+\kappa)f\|_{L^2} 
\end{equation*}
for all $1\leq n\leq \infty$. 
 Therefore, it follows from the Kato--Rellich theorem \cite{ReedSimonvol2}*{Th.~X.12}, that $L_n$ is a self-adjoint and semi-bounded operator on the domain $H^1_+$. In fact, we once again have the stronger result that $\Pi_nu\Pi_n\ol{u}\Pi_n$ is infinitesimally small with respect to the operator $L_0 = -i\partial$ on $L^2_+$. 

 Therefore, it follows from the above inequalities that 
\begin{equation*}
\tfrac{1}{2} \| (L_0+\kappa)f\|_{L^2}
\leq \| (L_n+\kappa)f\|_{L^2}
\leq \tfrac{3}{2} \| (L_0+\kappa)f\|_{L^2} ,
\end{equation*}
which proves~\eqref{L'}.  The estimate \eqref{R'} then follows by duality.
\end{proof}

\subsection{Norm resolvent convergence}\label{sec:norm-res-cv}

\begin{lem}[Norm resolvent convergence for \eqref{eq:BO}]
\label{lem:resolv 2}
Given a real-valued function $u\in L^2$, the operators $L_n$ converge to $L_\infty$ in norm resolvent sense as $n\to\infty$.
\end{lem}
\begin{proof}
Let $\kappa_0$ denote the constant from Proposition~\ref{prop:lax}.
It suffices to show that for any $\kappa\geq\kappa_0$, we have
\begin{equation*}
\| R_n(\kappa) - R_\infty(\kappa) \|_{\op} \xrightarrow{n \rightarrow \infty} 0  .
\end{equation*}
  Using the resolvent identity 
  \[R_n - R_{\infty} = R_{n}(L_\infty - L_n )R_{\infty} = - R_n (\Pi u\Pi  - \Pi_nu\Pi_n)R_\infty ,
  \]
  we have
  \[
  \| R_n - R_{\infty}\|_{\op} \le \| R_{n}R_0^{-1}\|_{\op} \| R_0(\Pi u\Pi  - \Pi_nu\Pi_n)R_0\|_{\op} \| R_{0}^{-1}R_{\infty}\|_\op.
  \]
  Using inequality \eqref{R}, for $f\in L^2_{+}$ we have
  \[
  \| R_{n}R_0^{-1}f \|_{L^2} \le 2 \| R_0^{-1}f\|_{H^{-1}_{\kappa}} = 2\| f\|_{L^2},
  \]
  hence the first factor is bounded by $ \| R_{n}R_0^{-1}\|_{\op} \le 2$.
  Next, applying \eqref{L} in the case $n=\infty$, we have
  \[
  \| R_{0}^{-1}R_{\infty} f\|_{L^2} = \| R_{\infty} f\|_{H^1_{\kappa}} \le 2\|(L_{\infty} + \kappa) R_{\infty} f \|_{L^2} = 2 \| f\|_{L^2},
  \]
  hence the last factor is bounded by $\| R_{0}^{-1}R_{\infty}\|_\op \le 2$.
Therefore, to finish the proof, it suffices to show
\begin{equation*}
\| R_0(\Pi u\Pi - \Pi_nu\Pi_n) R_0\| _{\op} \le \| R_0(\Pi-\Pi_n) u \Pi R_0 \|_{\op}  + \| R_0\Pi_nu (\Pi-\Pi_n) R_0 \|_{\op}
\end{equation*}
converges to zero as $n \rightarrow \infty$.

For the first term, using \eqref{eq:decay} and \eqref{est 1} with $n=\infty$ we have
\[
\| R_0(\Pi-\Pi_n) u \Pi R_0 \|_{\op} \le \| R_0(\Pi-\Pi_n)\|_\op \| u \Pi R_0 \|_{\op} \le \frac{1}{n}\sqrt{3}\kappa^{-1/2}\| u\|_{L^2}.
\]

Next, recall that for a bounded operator $B$ on a Hilbert space, we have
\begin{equation}
\|B^*\|_{\op} = \|B\|_{\op}.
\label{B*B}
\end{equation}
Thus, the second term is bounded by
\begin{align*}
\| R_0\Pi_nu (\Pi-\Pi_n) R_0 \|_{\op} &\le \| R_0\Pi_n(u \ \cdot) \|_{\op}\|(\Pi-\Pi_n) R_0 \|_{\op}\\
& = \| u \Pi_n R_0\|_{\op}\| R_0 (\Pi-\Pi_n)\|_{\op}\\
& \le \frac{1}{n}\sqrt{3}\kappa^{-1/2}\| u\|_{L^2},
\end{align*}
where we again used estimates \eqref{est 1} and \eqref{eq:decay}. In conclusion,
\[
\| R_n(\kappa) - R_\infty(\kappa) \|_{\op} \le \frac{1}{n}8\sqrt{3}\kappa^{-1/2}\| u\|_{L^2} \xrightarrow{n \to \infty} 0. \qedhere
\]

\end{proof}

\begin{lem}[Norm resolvent convergence for \eqref{eq:CS}]
\label{lem:resolv 2'}
Given $u\in L^2_+$, the operators $L_n$ converge to $L_\infty$ in norm resolvent sense as $n\to\infty$.
\end{lem}
\begin{proof}

Let $\kappa_0$ denote the constant from Proposition~\ref{prop:lax'}.
Following the proof of Lemma \ref{lem:resolv 2}, it suffices to show that for any $\kappa\geq\kappa_0$, we have
  \[
  \| R_n - R_{\infty}\|_{\op} \le \| R_{n}R_0^{-1}\|_{\op} \|R_0 (\Pi u\Pi\ol u \Pi  - \Pi_nu\Pi_n \ol u \Pi_n)R_0\|_{\op} \| R_{0}^{-1}R_{\infty}\|_\op
  \]
  goes to zero as $n\to \infty$.
  
  Using inequalities \eqref{L'} and \eqref{R'}, we again have
  \[
   \| R_{n}R_0^{-1}\|_{\op} \le 2 \quad \text{and} \quad \| R_{0}^{-1}R_{\infty}\|_\op \le 2.
  \]
Therefore, to finish the proof, it suffices to show
\begin{align*}
\| R_0 (\Pi u\Pi\ol u \Pi  - \Pi_nu\Pi_n \ol u \Pi_n)R_0 \|_{\op} &\le  \|R_0(\Pi-\Pi_n) u \Pi \ol u \Pi R_0 \|_{\op} \\
&+ \| R_0\Pi_n u (\Pi-\Pi_n) \ol u \Pi R_0 \|_{\op}\\
&+\| R_0\Pi_n  u \Pi_n \ol u (\Pi-\Pi_n) R_0 \|_{\op}
\end{align*}
goes to zero as $n \rightarrow \infty$.
For the first term, using estimates \eqref{eq:decay} and \eqref{est 4}  we have
\[
\|R_0(\Pi-\Pi_n) u \Pi \ol u \Pi R_0 \|_{\op} \le \frac{1}{n}\| u \Pi \ol u \Pi R_0\|_{\op} \le \frac{2\| u\|^2}{n}
\]
which goes to zero as $n \to \infty$. Moreover, as $u\in L^2_+$, the operator $\Pi_n u(\Pi-\Pi_n)$ is zero, hence the second term vanishes.
Finally, for the third term, by noting that $(R_0\Pi_n  u \Pi_n \ol u)^* = u \Pi_n \ol u \Pi_n R_0$ and applying equation \eqref{B*B} we have
\[
\| R_0\Pi_n  u \Pi_n \ol u (\Pi-\Pi_n) R_0 \|_{\op} \le \| u \Pi_n \ol u \Pi_n R_0\|_{\op}\| (\Pi-\Pi_n) R_0 \|_{\op} 
\le \frac{2\| u\|^2}{n},
\]
where we once again used \eqref{eq:decay} and \eqref{est 4}. This finishes the proof.
\end{proof}

\subsection{The generated unitary groups}\label{sec:cv-group}

From Section \ref{sec:adj-equiv} we have that the truncated Lax operators $L_n$ for \eqref{eq:BO} and \eqref{eq:CS} are self-adjoint.  By the functional calculus, it follows that each of these operators generates a strongly continuous one-parameter unitary group $t\mapsto e^{itL_n}$ on $L^2_{+}$, see \cite{ReedSimonvol1}*{Th.~VIII.7}.

To conclude this section, we will prove the following result regarding the convergence of these unitary groups as $n\to\infty$.  A key ingredient of the proof are Lemmas~\ref{lem:resolv 2} and~\ref{lem:resolv 2'}, which establish that $L_n$ converges to $L$ in norm resolvent sense.  It is well-known that this implies $e^{itL_n}\to e^{itL}$ in the strong operator topology, for each time $t$, see for example \cite{ReedSimonvol1}*{Th.~VIII.20}.  By comparison, our result~\eqref{resolv} states that for the sequence of truncated Lax operators $L_n$, the convergence of the unitary groups $e^{itL_n}$ is uniform on both {\it compact time intervals} and {\it compact subsets of $L^2_+$}.
\begin{prop}\label{lem:resolv}
For $1\leq n\leq \infty$, let $L_n$ denote the truncated Lax operators for \eqref{eq:BO} or \eqref{eq:CS} constructed in Proposition~\ref{prop:lax} or~\ref{prop:lax'} for any fixed initial data $u$.  Then, for any $T>0$ and any compact set $\mc{F}\subset L^2_{+}$, we have
\begin{equation}
\sup_{f\in \mathcal{F}}\ \sup_{t\in[-T,T]}\, \| (e^{itL_n} - e^{itL_\infty})f \|_{L^2} \xrightarrow{n \to \infty} 0.
\label{resolv}
\end{equation}
\end{prop}

\begin{proof}
We will present the argument in the settings of~\eqref{eq:BO} and~\eqref{eq:CS} in parallel, indicating the source of each ingredient for both cases when necessary.

Set $\kappa = \kappa_0$, where $\kappa_0 = \kappa_0(u)$ is defined in Proposition~\ref{prop:lax} and~\ref{prop:lax'} for~\eqref{eq:BO} and~\eqref{eq:CS} respectively.
For $M \ge 1$, let $\chi_{\leq M}: \mathbb{R} \to  [0,1]$ be a smooth function that is equal to one on $[-\kappa, M]$, and goes to zero as $x\to +\infty$. We also let $X = [-T,T] \times [ -\kappa, +\infty)$ and define
\[ g \colon
\begin{array}{rccl}
 &X & \longrightarrow & \mathbb{R} \\
&(t, x) & \longmapsto & e^{itx} \chi_{\leq M}(x).
\end{array}\]

By the Stone--Weierstrass theorem, polynomials in $t$ and $(x+\kappa)^{-1}$ are dense in $C_{0}(X)$.  Therefore, for any $\epsilon>0$ there exists a polynomial $P(t,y)$ 
such that
\[
\sup_{(t,x)\in X} \left|P\big(t,(x+\kappa)^{-1}\big) - g(t,x)\right| \le \epsilon.
\] 
Hence, by the functional calculus we have
\begin{equation}\label{eq:secondT}
\sup_{\substack{t \in [-T,T] \\[.2em] 1\le n \le \infty }} \left\| P\big(t, (L_n + \kappa)^{-1}\big) - g(t, L_n)\right\|_{\op} \le \epsilon.
\end{equation}

We can then decompose the error into three terms,
\begin{align}\label{eq:3terms}
\| (e^{itL_n} - e^{itL_\infty})f \|_{L^2}\nonumber
{}\leq{}& 2\sup_{\substack{t \in [-T,T] \\[.2em] 1\le n \le \infty }}\| e^{itL_{n}}f - g(t, L_n)f \|_{L^2} \\
&+ 2C\sup_{\substack{t \in [-T,T] \\[.2em] 1\le n \le \infty }} \big\|  P\big(t, (L_n+\kappa)^{-1}\big) - g(t, L_n) \big\|_{\op} \\\nonumber
&+ C\sup_{t\in[-T,T]}\big\|  P\big(t, (L_n+\kappa)^{-1}\big) - P\big(t,(L+\kappa)^{-1}\big) \big\|_{\op},
\end{align}
with $C = \sup_{f\in \mathcal{F}} \| f\|_{L^2}$. The second term in \eqref{eq:3terms} can be taken arbitrarily small using \eqref{eq:secondT}. We now prove that the third term in \eqref{eq:3terms} goes to zero as $n\to \infty$. Indeed, for $m\in \mathbb{N}$ we have
\begin{align*}
R_n^{m} - R_\infty^{m} = R_n^{m-1} &\left( R_n - R_\infty \right) 
+\left(R_n^{m-1} - R_\infty^{m-1} \right)R_\infty,
\end{align*}
and hence, using \eqref{R} and \eqref{R'},
\[
\|R_n^{m} - R^{m}_\infty \|_\op \le 2^{m-1}\|R_n - R_\infty \|_\op + 2
\|R_n^{m-1} - R_\infty^{m-1}\|_\op.
\]
By induction, $\|R_n^{m} - R^{m}_\infty \|_\op \le 3^m \|R_n - R_\infty \|_\op$, which goes to zero as $n\to \infty$ by Lemmas~\ref{lem:resolv 2} and \ref{lem:resolv 2'}.
Next, by writing the polynomial as $P(t,y) = \sum_{\ell,m}\alpha_{\ell,m} t^{\ell} y^{m}$, it follows that
\[
\sup_{t\in[-T,T]}\|  P\big(t, (L_n+\kappa)^{-1}\big) - P\big(t,(L+\kappa)^{-1}\big)\|_{\op} \le \sum_{\ell, m} \alpha_{\ell,m} T^{\ell} \|R_n^{m} - R_\infty^{m} \|_\op \xrightarrow{n\to\infty} 0 .
\]
Finally, we show that the first term in \eqref{eq:3terms} goes to zero as $M \to \infty$, finishing the proof. As $\mathcal{F} \subset L^2_+$ is compact and $H^1_{+}$ is dense in $L^2_{+}$, we have that given any $\epsilon>0$ there exists $m\in\mathbb{N}$ and $\phi_1, \dots, \phi_m \in H^1_{+}$ such that 
\[
\forall f\in\mathcal{F}, \ \exists j \in \{1,\dots, m\},\ \|\phi_j - f\|_{L^2} \le \epsilon.
\]
Thus, by letting $\chi_{>M} = 1-\chi_{\le M}$ we have that for $t\in[-T,T]$ and $n\in\mathbb{N}$
\[
\| e^{itL_{n}}f - g(t, L_n)f \|_{L^2} \le \| \chi_{> M}(L_n)e^{itL_n}f\|_{L^2} \le \| \chi_{>M}(L_n) (f-\phi_i) \|_{L^2} + \| \chi_{>M}(L_{n})\phi_i\|_{L^2}.
\]
The first term in the right-hand side above is bounded by $\epsilon$ while for the second term, using the functional calculus and inequalities \eqref{R} and \eqref{R'}, we have
\[
\| \chi_{>M}(L_n) \phi_i\|_{L^2} \le \left\| \chi_{>M}(L_n)\frac{L_n+\kappa}{M+\kappa}\phi_i\right\|_{L^2}
\le \frac{1}{M} \| (L_n+ \kappa)\phi_i\|_{L^2}
\le \frac{3}{2M} \max_{1\le i\le m}\| \phi_i\|_{H^1_{\kappa}},
\]
which is independent of $n$ and goes to zero as $M\to \infty$.
To conclude the proof we take first $M$, then $\deg(P)$, and finally $n$ large enough, in order for the three terms in \eqref{eq:3terms} to be arbitrarily small. \qedhere

\end{proof}


\subsection{Proof of Theorem \ref{thm:main}}\label{sec:pfThms}
We divide the proof into two steps. First, we use Proposition~\ref{lem:resolv} to show weak $L^2$-convergence of the scheme $u_{K}(t)$ to the solution $u(t)$, uniformly for $t\in[-T,T]$. We then leverage to strong convergence using the boundedness of the $L^2$-norm of the scheme, see Proposition \ref{prop:L2 consv}.

\begin{proof}[Proof of Theorem \ref{thm:main}]
We will consider both the~\eqref{eq:BO} and~\eqref{eq:CS} cases simultaneously.  Indeed, the primary input will be Propositions~\ref{prop:L2 consv} and~\ref{lem:resolv}, which apply to both cases.

Fix $T>0$ and let $u_0\in L^2(\mathbb T)$ for \eqref{eq:BO}, and $u_0\in L^2_+$ in the case of~\eqref{eq:CS} with the additional assumption $\| u_0\|_{L^2} < 1$ in the focusing case. 
Let $K>0$ be the highest frequency in the approximation and $(n(k))_{k\ge 0}\in\mathbb N^{\mathbb{N}}$ be a sequence such that
\[
n(k)\xrightarrow[K\to\infty]{}\infty , \quad \text{for  each} \ k\ge 0.
\]

We begin by comparing the scheme $ u_K $ to the sequence $ w_K $, which is generated by the same recursion \eqref{eq:newScheme BO} as $ u_K $, but initialized with the seed $ w^0 = \Pi u_0 $ instead of $ u^0 = \Pi_{n(0)} u_0 $.
Namely, for $t \in [-T,T]$ and $0\leq k<K$ we define the Fourier coefficients of $w_K$ as
\begin{equation*}
\wh{w}_K(t,k) = \big\langle (e^{\pm it(1+2L_{n(k)})}S^{*}\big) \cdots \big(e^{\pm it(1+2L_{n(2)})}S^{*}\big) \big(e^{\pm it(1+2L_{n(1)})}S^{*}\big) \Pi u_0 , 1 \big\rangle 
\end{equation*}
and we set $\wh{w}_K(t,k) = 0$ for $k\geq K$. The above holds with a positive sign (+) for \eqref{eq:BO} and a negative sign (--)  for \eqref{eq:CS}.  For $j\in\mathbb{N}$, each operator $e^{itL_j}$ is unitary and $\|S^{*} \|_\op\le 1$, hence by Cauchy--Schwarz we have
\begin{align*}
\big| \widehat{w}_{K}(t,k) - \widehat{u}_K(t,k) \big| 
&= \big| \big\langle (e^{\pm it(1+2L_{n(k)})}S^{*}\big) \cdots \big(e^{\pm it(1+2L_{n(1)})}S^{*}\big) (\Pi - \Pi_{n(0)}) u_0 , 1 \big\rangle \big| \\
&\le \big\| (e^{\pm it(1+2L_{n(k)})}S^{*}\big) \cdots \big(e^{\pm it(1+2L_{n(1)})}S^{*}\big) (\Pi - \Pi_{n(0)}) u_0 \big\|_{L^2} \\
&\leq \| (\Pi - \Pi_{n(0)}) u_0 \big\|_{L^2} .
\end{align*}
The above right-hand side converges to zero as $n(0)\to\infty$. 
Therefore, for each $k\geq 0$, we have
\begin{equation}
|\widehat{w}_{K}(t,k) - \wh{u}_K(t,k)| \xrightarrow{ K\to\infty} 
0, \quad \text{ uniformly for }t\in [-T,T] .
\label{vk to wk}
\end{equation}
In view of this, it will suffice to study the behavior of the Fourier coefficients of $w_K$ rather than $u_K$.

For $k \ge 0$, we write
\begin{align*}
\widehat{w}_{K}(t,k) - \widehat{u}(t,k) = \sum_{j=1}^{k} \Big\langle Z_{k,j} \ \big( e^{\pm it(1+2L_{n(j)})} - e^{\pm it(1+2L)}\big) S^{*} (e^{\pm it(1+2L)}S^*)^{j-1}\Pi u_0 , 1 \Big\rangle,
\end{align*}
with $Z_{k,k} = \rm{Id}$ and $Z_{k,j-1} = Z_{k,j} \ e^{\pm it(1+2L_{n(j)})}S^{*}$. In particular, we have 
\[Z_{k,j} = e^{\pm it(1+2L_{n(k)})}S^{*}\cdots e^{\pm it(1+2L_{n(j+1)})}S^{*}, \quad j<k-1.
\]
Hence, the following bound holds
\begin{align*}
\big| \widehat{w}_{K}(t,k) - \widehat{u}(t,k) \big| \le \sum_{j=1}^{k} \big\| \big( e^{\pm 2itL_{n(j)}} - e^{\pm 2itL}\big) S^{*} (e^{\pm 2itL}S^*)^{j-1}\Pi u_0 \big\|_{L^2} .
\end{align*}
Next, notice that the operator 
\[ \begin{array}{rccl}
e^{itL} \colon &\mathbb{R} \times L^2_+ & \longrightarrow & L^2_+ \\
&(t, f) & \longmapsto & e^{itL} f
\end{array}\]
is jointly continuous.  Indeed, given a sequence $t_m\to t$ in $\R$ and $\phi_m\to \phi$ in $L^2_+$, we have
\begin{equation*}
\| e^{it_mL} \phi_m - e^{itL}\phi \|_{L^2}
\leq \| e^{it_m L} \|_{\op} \| \phi_m - \phi \|_{L^2} + \| e^{it_m L} \phi - e^{itL}\phi \|_{L^2} .
\end{equation*}
The right-hand side above tends to zero as $m \to \infty$, due to the bound $\|e^{it_m L}\|_{\mathrm{op}} = 1$ and the fact that the mapping $t \mapsto e^{itL}$ is strongly continuous (see for example \cite{ReedSimonvol1}*{Th.~VIII.7}).

Therefore, by composition of continuous applications, the mapping
\[ \begin{array}{ccl}
\mathbb{R} &\longrightarrow &L^2_{+}\\
t &\longmapsto& S^{*} (e^{\pm 2itL}S^*)^{j-1}\Pi u_0
\end{array}\]
is continuous for each $j = 1,\dots,k$. Hence,
for any $T>0$, the set
\begin{equation*}
\mc{F}_T = \big\{ S^{*} (e^{\pm 2itL}S^*)^{j-1}\Pi u_0 : t\in [-T,T],\ j=1,\dots, k \big\} 
\end{equation*}
is compact in $L^2_+$.  Applying Proposition~\ref{lem:resolv} to this set $\mc{F}_T$,
we conclude that for each $k\geq 0$ we have
\begin{equation*}
\widehat{w}_{K}(t,k) \xrightarrow{K\to\infty}
\widehat{u}(t,k) \quad\text{ uniformly for }t\in [-T,T] .
\end{equation*}
Comparing this with \eqref{vk to wk}, we see that for each $ 0 \le k < K$,
\begin{equation*}
\widehat{u}_{K}(t,k) \xrightarrow{K\to\infty}
\widehat{u}(t,k) \quad \text{ uniformly for }t\in [-T,T] .
\end{equation*}
In the \eqref{eq:BO} case, we use
the fact that the solution is real-valued and the definition of the scheme to see that the above convergence also holds for $-K < k\le 0$.

As a consequence, 
we conclude weak $L^2$ convergence:
\begin{equation}
u_{K}(t) \xrightharpoonup{K\to\infty}
u(t) \quad\text{ uniformly for }t\in [-T,T] .
\label{L2 wk}
\end{equation}
Indeed, by Proposition \ref{prop:L2 consv} the scheme is uniformly bounded and hence
\[
\|u(t) - u_K(t) \|_{L^2} \le \|u(t) \|_{L^2} + \| u_K(t)\|_{L^2} \le 2 \| u_0\|_{L^2}.
\]
Thereby, given $\epsilon > 0$ and a test function $\phi\in L^2$, for $ K_0 \in \mathbb{N}$ such that $ \sum_{|k|\ge K_0} |\widehat\phi(k)|^2 \le \epsilon^2$, we have
\[
|\langle u(t) - u_K(t), \phi\rangle| \le \sum_{|k|< K_0} |\widehat{\phi}(k)| \ |\widehat{u}(t,k) - \widehat{u}_{K}(t,k)| + 2\epsilon \| u_0\|_{L^2},
\]
where the first term can be made arbitrarily small, uniformly for $t\in [-T,T]$, by taking $K > K_0$ and $n(j)$ large enough for $0\le j< K_0$. This proves~\eqref{L2 wk}.

To finish the proof, we upgrade this to strong convergence:
\begin{equation}
\sup_{t\in[-T,T]} \| u_K(t) - u(t) \|_{L^2} \xrightarrow{K\to\infty}
0.
\label{L2 st 1}
\end{equation}
Using the boundedness of the $L^2$-norm of the numerical scheme \eqref{eq:l2-est}, we have
\begin{equation}\label{L2 st 2}
\begin{aligned}
\| u_K(t) - u(t) \|_{L^2}^2 
&= \| u_K(t) \|_{L^2}^2 + \| u(t) \|_{L^2}^2 - 2\Re\langle u_K(t), u(t) \rangle 
 \\
&\le 2 \| u(t) \|_{L^2}^2- 2\Re\langle u_K(t), u(t) \rangle\\
&= 2\Re \langle u(t) - u_K(t), u(t) \rangle.
\end{aligned}
\end{equation}
By the global well-posedness results for \eqref{eq:BO} and \eqref{eq:CS} discussed in the introduction, the mappings
\[ \eqref{eq:BO} : \begin{array}{ccl}
\mathbb{R} &\longrightarrow &L^2\\
t &\longmapsto& u(t)
\end{array} \quad \text{and} \quad
\eqref{eq:CS} : \begin{array}{ccl}
\mathbb{R} &\longrightarrow &L^2_{+}\\
t &\longmapsto& u(t)
\end{array}
\]
are continuous and thereby the orbits $\{ u(t) : t\in [-T,T] \}$ are compact in $L^2$ and $L^2_{+}$ respectively.  Therefore, given $\del>0$, there exist finitely many points $t_1,\dots,t_M \in [-T,T]$, such that the orbit can be covered by open balls in $L^2$, of radius $\del>0$, centered at $u(t_1),\dots,u(t_M)$.  Using the weak limit in equation \eqref{L2 wk}, for each $m=1,\dots,M$, we have
\begin{equation*}
\sup_{[-T,T]}\big| \langle u(t) - u_K(t), u(t_m) \rangle \big| \xrightarrow{K\to\infty}
0.
\end{equation*}
It follows from the above that given any $t\in [-T,T]$  and $\delta> 0$ there exists $m\in \{1,\dots,M\}$ such that $\| u(t) - u(t_m)\| \le \delta$ and hence
\begin{equation*}
\limsup_{K\to\infty} 
\big|\langle u(t) - u_K(t), u(t) \rangle \big| 
\leq \limsup_{K\to\infty}
\big| \langle u(t) - u_K(t), u(t_m) \rangle \big| + 2\del \| u(t) \|_{L^2} \\
= 2\delta\| u_0\|_{L^2},
\end{equation*}
uniformly for $t\in[-T,T]$.
The above holds for any $\del>0$, which allows us conclude that
\begin{equation*}
\sup_{t\in[-T,T]} \big| \langle u_K(t) - u(t), u(t) \rangle \big| \xrightarrow{K\to\infty}
0.
\end{equation*}
In view of the estimate \eqref{L2 st 2}, this implies the desired strong convergence \eqref{L2 st 1} of the schemes in $L^2$.
\end{proof}

\bibliographystyle{agsm}
\bibliography{refs}

\end{document}